\documentclass[final,1p,times]{elsarticle}

\usepackage[colorlinks=true]{hyperref}
\usepackage{amsfonts}
\usepackage{mathrsfs}
\usepackage{amsthm,amsmath,amssymb,mathrsfs,amsfonts,graphicx,graphics,latexsym,exscale,cmmib57,dsfont,bbm,amscd}
\usepackage{color,soul}
\newtheorem{thm}{\quad\ Theorem}[section]
\newtheorem{lem}[thm]{\quad\ Lemma}

\newtheorem*{cor}{\quad Corollary}
\newtheorem{corr}{\quad Corollary}
\newtheorem{corollary}{\quad Corollary}

\theoremstyle{definition}
\newtheorem{defn}[thm]{\quad\ Definition}
\newtheorem{notation}[thm]{\quad\ Notation}
\newtheorem*{note}{\quad\ Note}
\journal{xxx}

\begin{document}

\begin{frontmatter}
\title{When is weakly almost periodic equivalent to uniformly almost periodic for semiflows?}

\author{Xiongping Dai}
\ead{xpdai@nju.edu.cn}
\address{Department of Mathematics, Nanjing University, Nanjing 210093, People's Republic of China}

\begin{abstract}
Let $\pi\colon T\times X\rightarrow X$ be a semiflow on a compact Hausdorff space $X$ with phase semigroup $T$.
This paper provides us with some conditions under which \textsl{weakly almost periodic} is equivalent to \textsl{uniformly almost periodic} or to \textsl{equicontinuous} for $(T,X,\pi)$.
\end{abstract}

\end{frontmatter}

\section{Introduction}\label{sec1}
Throughout this paper, let $X$ be a compact T$_2$-space with the symmetric uniform structure $\mathscr{U}_X$, unless stated otherwise. Given any $x\in X$ and $\varepsilon\in\mathscr{U}_X$, write $\varepsilon[x]=\{y\in X\,|\,(x,y)\in\varepsilon\}$. By $C(X,X)$ we denote the set of all continuous mappings of $X$ to itself, and by $C(X,\mathbb{R})$ the set of all real continuous functions defined on $X$.

As usual, a \textit{semiflow} is a triple $(T,X,\pi)$, or simply write $(T,X)$, where $T$ is a topological semigroup with a neutral element $e$ (i.e. $et=te=t\ \forall t\in T$), and where the action of the phase semigroup $T$ from the left on the phase space $X$, $\pi\colon T\times X\rightarrow X,\ (t,x)\mapsto tx$
is a jointly continuous mapping such that
\begin{itemize}
\item $ex=x\ \forall x\in X$ (i.e. $\pi_e=\textit{id}_X$) and $(ts)x=t(sx)\ \forall s,t\in T, x\in X$.
\end{itemize}
If $T$ is just a topological group here, then $(T,X)$ will be called a \textit{flow} with phase group $T$.

\begin{notation}
Given any $x\in X$, by $Tx$ we will denote the orbit $\{tx\,|\,t\in T\}$ of $x$ under a semiflow $(T,X)$. Let $\textrm{cls}_XTx$ stand for the closure of $Tx$ in $X$.

For any $f\in C(X,\mathbb{R})$, we shall write $Tf=\{f_t\,|\,t\in T\}$, where
$f_t\colon X\rightarrow\mathbb{R}$ is defined by $x\mapsto f(tx)$, associated to $(T,X)$.
\end{notation}

\begin{notation}[\cite{Fur}]
A subset $A$ of $T$ is said to be \textit{syndetic} if there is a compact subset $K$ of $T$ such that $Kt\cap A\not=\emptyset$ for each $t\in T$. This is equivalent to condition $T=KA$ for some compact subset $K$ of $T$, whenever $T$ is a group.
\end{notation}

\begin{defn}[{cf.~\cite{Aus,DX,AD}}]\label{def1.3}
A semiflow $(T,X)$ is called \textit{uniformly almost periodic} (for short: \textit{u.a.p.}) in case for any $\varepsilon\in\mathscr{U}_X$, there is a syndetic subset $A$ of $T$ such that $Ax\subseteq\varepsilon[x]$ for all point $x$ of $X$.
\end{defn}

By \cite[Proposition~4.15]{E69}, a \textit{flow} $(T,X)$ is \textit{u.a.p.} if and only if $Tf$ is relatively compact in $C(X,\mathbb{R})$ under the uniform convergence topology.

\begin{defn}[{cf.~\cite{DX,AD}}]\label{def1.4}
A semiflow $(T,X)$ is called \textit{equicontinuous} in case given any $\varepsilon\in\mathscr{U}_X$, there exists a $\delta\in\mathscr{U}_X$ such that $T\delta\subseteq\varepsilon$; i.e., $(x,y)\in\delta$ implies $(tx,ty)\in\varepsilon\ \forall t\in T$, for all $x,y\in X$.

It should be noted that in general, the equicontinuous is conceptually weaker than the uniformly almost periodic, for semiflows. See, e.g., \cite{AD,DX}, for counterexamples.
\end{defn}

\begin{defn}[{cf.~\cite{EN} for $T$ a group}]\label{def1.5}
A semiflow $(T,X)$ is referred to as \textit{weakly almost periodic} (\textit{w.a.p.} for short) if given any $f\in C(X,\mathbb{R})$, $Tf$ is relatively compact in the topology of pointwise convergence on $C(X,\mathbb{R})$.

Recall that for a net $\{f_n\}$ and $g$ in $C(X,\mathbb{R})$, $f_n\to g$ under the pointwise convergence topology if and only if $f_n(x)\to g(x)$ for all $x\in X$.
\end{defn}

According to Definition~\ref{def1.5}, we can easily obtain the following results and we omit their proofs here.
\begin{itemize}
\item Let $(T,X)$ be a \textit{w.a.p.} semiflow and $\Lambda$ an invariant closed subset of $X$; then $(T,\Lambda)$ is a \textit{w.a.p.} subsemiflow of $(T,X)$.

\item If $(T,X)$ is an equicontinuous semiflow, then for any $f\in C(X,\mathbb{R})$, $Tf$ is an equicontinuous family of functions. Thus by Ascoli's theorem, $Tf$ is relatively compact in $C(X,\mathbb{R})$ under the pointwise topology so $(T,X)$ is \textit{w.a.p.} by Definition~\ref{def1.5}.
\end{itemize}

\begin{thm}[{\cite[Theorem~1.5]{DX}}]\label{thm1.6}
Let $(T,X)$ be a semiflow. Then $(T,X)$ is uniformly almost periodic if and only if it is equicontinuous such that each $t\in T$ is a self-surjection of $X$.
\end{thm}

Thus, \textit{u.a.p.} $\Rightarrow$ equicontinuous $\Rightarrow$ \textit{w.a.p.}, for any semiflow on the compact T$_2$-space $X$. It is well known that their converses are false in general.
In this paper, we will be mainly concerned with the following basic question:
\begin{quote}
 \textit{When does `weakly almost periodic' imply `equicontinuous' or `uniformly almost periodic' and imply `uniquely ergodic' for a semiflow on $X$?}
\end{quote}

See Theorems~\ref{thm2.15A}, \ref{thm2.16A}, \ref{thm2.27}, \ref{thm2.31A}, \ref{thm2.32A}, and \ref{thm2.34A} for the first part of this question, and Theorems~\ref{thm3.5} and \ref{thm3.6} for the second part.
\section{Uniform almost periodicity of semiflows}\label{sec2}
In this section we will provide a series of conditions under which weakly almost periodic implies equicontinuous or uniformly almost periodic, mainly based on the recent work \cite{AD}.

Moreover, we will generalize a theorem of Gottschalk~\cite{G} from abelian group action to abelian semigroup case in $\S\ref{sec2.12}$, by using completely different approaches.

\subsection{Ellis' semigroup}\label{sec2.1}
Let $X^X$ be the set of all functions from $X$ to itself, continuous or not, which is a compact
T$_2$-space under the topology $\mathfrak{p}$ of pointwise convergence as follows: a net $\{p_n\}$ in $X^X$ converges to $q$ if and only if $p_n(x)\to q(x)$ for each $x\in X$. A subbase for $\mathfrak{p}$ is the family of all subsets of the form $\{f\,|\,f(x)\in U\}$, where $x$ is a point of $X$ and $U$ is open in $X$. See \cite{E69,Fur}.

The space $X^X$  is naturally provided with a semigroup structure: if $p,q\in X^X$ then $pq\colon X\rightarrow X$ defined by $x\mapsto p(qx)$ is the composition of $p$ and $q$.

\begin{defn}[{cf.~\cite{E69, Fur, Aus}}]
Let $(T,X,\pi)$ be a semiflow. Then for any $t\in T$, $\pi_t$ is a continuous self-map of $X$, hence an element of the space $X^X$.
\begin{enumerate}
\item The \textit{Ellis semigroup} $E(T,X)$, or simply $E(X)$, of $(T,X,\pi)$ is the closure of the semigroup $\{\pi_t\,|\,t\in T\}$ in $X^X$ with respect to the topology $\mathfrak{p}$.

\item An element $u\in E(X)$ is called an \textit{idempotent} in $E(X)$ if $u^2=u$.
\end{enumerate}

Since $E(X)$ is a compact right-topological semigroup, i.e., $E(X)$ is a semigroup and a compact T$_2$-space such that $R_q\colon p\mapsto pq$ is continuous under $\mathfrak{p}$ for any $q\in E(X)$, hence there always exist idempotents in $E(X)$ (cf.,~e.g.,~\cite{Fur} and \cite[Lemma~6.6]{Aus}).
\end{defn}

\begin{defn}[\cite{Fur, AD}]\label{def2.2}
A point $x$ of $X$ is called a \textit{distal point} for $(T,X)$ if there exists no point $y$ of $\textrm{cls}_XTx$, $y\not=x$, such that there is a net $\{t_n\}$ in $T$ with $\lim t_nx=\lim t_ny$. Further if $(T,X)$ is pointwise distal, then $(T,X)$ is called a \textit{distal semiflow}.

It is easy to check that $(T,X)$ is distal \textit{iff} given any $x,y\in X, x\not=y$, there is an $\varepsilon\in\mathscr{U}_X$ such that $(tx,ty)\not\in\varepsilon$ for each $t\in T$.
\end{defn}

Following the framework of Robert~Ellis, the distal has the following characterization by using Ellis' semigroup for semiflows:

\begin{lem}[{\cite[Lemma~3.8]{AD}}]\label{lem2.3}
A semiflow $(T,X)$ is distal if and only if its Ellis semigroup $E(X)$ is a group with the neutral element $\textit{id}_X$.
\end{lem}

Then for any flow, \textit{u.a.p.} can be characterized by equicontinuous and Ellis' semigroup as follows, (3) of which is equivalent to that $E(T,X)$ is a group of self-homeomorphisms of $X$.

\begin{thm}[{Ellis~\cite[Proposition~4.4]{E69}}]\label{thm2.4}
Let $(T,X)$ be a flow with phase group $T$; then the following statements are pairwise equivalent:
\begin{enumerate}
\item[$(1)$] $(T,X)$ is equicontinuous.
\item[$(2)$] $(T,X)$ is uniformly almost periodic.
\item[$(3)$] $E(T,X)$ is a group of continuous self-maps of $X$.
\end{enumerate}
\end{thm}

Condition (1) $\Leftrightarrow$ (2) of Theorem~\ref{thm2.4} has already been generalized to semiflows by the author of this paper and Xiao (cf.~Theorem~\ref{thm1.6} in $\S\ref{sec1}$).

The proof of Theorem~\ref{thm1.6} is based on the following theorem which itself will be needed as well as in our later arguments.

\begin{thm}[{\cite[Theorem~2.3]{AD}}]\label{thm2.5}
If $(T,X)$ is an equicontinuous semiflow with each $t\in T$ a self-surjection of $X$, then $(T,X)$ is a distal semiflow.
\end{thm}

The condition that each $t\in T$ is a self-surjection of $X$ is essential for this theorem (cf.~\cite{AD}). The case that $T$ is abelian has been suggested to the author in personal communications independently by E.~Akin and X.~Ye.

\subsection{Characterization of weak almost periodicity}\label{sec2.2}
The following Lemma~\ref{lem2.6} is just the semiflow version of Ellis and Nerurkar's \cite[Proposition~II.2]{EN} for flows, which characterizes \textit{w.a.p.} in terms of Ellis' semigroup.

Although the proof of Lemma~\ref{lem2.6} is almost same as that of the flow case, we will prove it here for reader's convenience.

\begin{lem}\label{lem2.6}
$(T,X)$ is a weakly almost periodic semiflow if and only if its Ellis semigroup $E(T,X)\subset C(X,X)$.
\end{lem}

\begin{proof}
Let $(T,X)$ be \textit{w.a.p.} and $p\in E(X)$. To show $p$ is continuous, it is enough to verify that for any $f\in C(X,\mathbb{R})$ the function $f\circ p\colon X\rightarrow\mathbb{R}$ is continuous. For this, let $t_n\to p$ in $E(X)$; then $f_{t_n}\to f\circ p$ in pointwise topology. Then by Definition~\ref{def1.5}, there is no loss of generality in assuming that $f_{t_n}\to g$, for some $g\in C(X,\mathbb{R})$, in the sense of the pointwise topology. Thus $f\circ p=g$ is continuous.

Conversely, let each $p\in E(X)$ be continuous. Let $f\in C(X,\mathbb{R})$ be any given and we define $\phi(p)=f\circ p\ \forall p\in E(X)$. Note that if $x_n\to x$ in $X$, then $\lim\phi(p)(x_n)=f(p(x))=\phi(p)(x)$. Thus $\phi(p)\in C(X,\mathbb{R})$. Now we claim that the map $\phi\colon E(X)\rightarrow C(X,\mathbb{R})$ is continuous under the pointwise topologies. To see this let $p_n\to p$ in $E(X)$; then
\begin{gather*}\lim\phi(p_n)(x)=\lim f(p_n(x))=f(p(x))=\phi(p)(x).\end{gather*}
Thus the image of $\phi$ is compact and it contains the set $Tf=\phi(T)$. Thus $(T,X)$ is a \textit{w.a.p.} semiflow.

The proof of Lemma~\ref{lem2.6} is thus completed.
\end{proof}

\begin{defn}[\cite{GH}]\label{def2.7}
Let $(T,X)$ be a semiflow with any phase space $X$ and with any phase semigroup $T$.
\begin{enumerate}
\item By $\textrm{Aut}\,(T,X)$ we denote the automorphism group of $(T,X)$; that is, $\textrm{Aut}\,(T,X)$ is the group of all self-homeomorphisms $a$ of $X$ such that $at=ta$ for all $t\in T$.

\item If $\textrm{Aut}\,(T,X)x=X$ for some $x\in X$ (so for all $x\in X$), then $(T,X)$ is called a \textit{universally transitive} semiflow (or \textit{algebraically transitive} in \cite{G,SKBS}).
\end{enumerate}
\end{defn}

\begin{lem}
Let $(T,X)$ be a universally transitive semiflow. If each $p\in E(T,X)$ has at least a continuous point, then $(T,X)$ is weakly almost periodic.
\end{lem}

\begin{proof}
Let $p\in E(T,X)$ which is continuous at a point $x_0\in X$ and let $T\ni t_n\to p$ under the pointwise topology. Then by $at_n=t_na$ for all $a\in\textrm{Aut}\,(T,X)$, $ap=pa$ for all $a\in\textrm{Aut}\,(T,X)$. Given any $x\in X$, there is some $a\in\textrm{Aut}\,(T,X)$ such that $x=ax_0$ by universal transitivity. Thus $p$ is continuous at $x$ and so $p\in C(X,X)$. This shows that $(T,X)$ is \textit{w.a.p.} by Lemma~\ref{lem2.6}.
\end{proof}
\subsection{Almost periodic points}
\begin{defn}[\cite{GH,E69,Fur,Aus}]
$(T,X)$ be a semiflow. Then
\begin{enumerate}
\item A point $x\in X$ is said be \textit{almost periodic} for $(T,X)$ if given any neighborhood $V$ of $x$, $N_T(x,V):=\{t\in T\,|\,tx\in V\}$ is syndetic in $T$.

\item[{}] If each point of $X$ is almost periodic, then $(T,X)$ is called a \textit{pointwise almost periodic} semiflow.

\item A non-empty closed subset of $X$ is called \textit{minimal} for $(T,X)$ if it is invariant and it contains no invariant closed proper subset of $X$.

\item[{}] Thus a non-empty subset of $X$ is minimal if and only if it is the orbit closure of each of its points.
\end{enumerate}
\end{defn}

It is a well-known fact that $x\in X$ is almost periodic for $(T,X)$ \textit{iff} it is a minimal point in the sense that $\textrm{cls}_XTx$ is a minimal subset of $(T,X)$ (cf., e.g.,~\cite[Lemma~1.5 in $\S1.1$]{AD}). Thus pointwise almost periodic is equivalent to pointwise minimal. See \cite{CD} for some other characterizations of almost periodic points of semiflows.

\subsection{Auslander's theorem}
In view of Lemma~\ref{lem2.6}, Ellis' theorem, condition (1)/(2) $\Leftrightarrow$ (3) of Theorem~\ref{thm2.4}, was already sharpened by Joseph Auslander as follows:

\begin{thm}[{Auslander~\cite[Theorem~4.6]{Aus}}]\label{thm2.10A}
Let $(T,X)$ be a minimal flow. Then $(T,X)$ is weakly almost periodic if and only if it is equicontinuous.
\end{thm}

Recall that a semigroup $T$ is called \textit{amenable} if every semiflow $(T,X)$ admits an invariant Borel probability measure on $X$. Clearly, every abelian semigroup is amenable.

\begin{lem}[{\cite[Proposition~3.19]{AD}}]\label{lem2.11A}
Let $(T,X)$ be a pointwise minimal semiflow with $T$ an amenable semigroup. Then each $t\in T$ is a surjection of $X$.
\end{lem}

A set $S\subseteq E(T,X)$ is called a \textit{topological semigroup} if $(f,g)\mapsto f\circ g$ of $S\times S$ to $S$ is a continuous map under the topology $\mathfrak{p}$ of pointwise convergence.

Further, in view of Theorem~\ref{thm1.6} and Lemma~\ref{lem2.6}, Auslander's theorem has recently been generalized as follows:

\begin{thm}[{\cite[Theorem~3.23]{AD}}]\label{thm2.12A}
Let $(T,X)$ be a semiflow with $T$ an amenable semigroup and with a dense set of almost periodic points. Then $(T,X)$ is uniformly almost periodic if and only if it is weakly almost periodic such that $E(T,X)$ is a topological semigroup.
\end{thm}

In this theorem, the condition that there is a dense set of almost periodic points is essentially needed (cf.~\cite[Example~3.25]{AD}).
\subsection{An algebraic lemma}
We will further generalize Auslander's theorem in this paper.
For that, we will need an algebraic lemma. In preparation, we will first recall a notion---minimal left ideal in a general semigroup.

\begin{defn}[\cite{E69,Aus}]
Let $E$ be an algebraic semigroup with a product operation.
\begin{enumerate}
\item A \textit{left ideal} in $E$ is a nonempty subset $I$ such that $EI\subseteq I$. A \textit{minimal left ideal} in $E$ is one which does not properly contain a left ideal.

\item By $J(I)$ we will denote the set of all idempotents in a left ideal $I$; that is, $u\in J(I)$ iff $u^2=u$ and $u\in I$.
\end{enumerate}

In our dynamics situation $J(I)$ is always non-empty for $I\subseteq E(T,X)$ is a compact right-topological semigroup.
\end{defn}

\begin{lem}[{cf.~\cite[Lemmas~6.1, 6.2 and 6.3]{Aus}}]\label{lem2.14A}
Let $E$ be any semigroup and let $I,I^\prime$ be minimal left ideals in $E$ with $J(I)\not=\emptyset$. Then:
\begin{enumerate}
\item[$(1)$] $Ip=I$ for all $p\in I$.
\item[$(2)$] $pu=p$ for all $u\in J(I), p\in I$.
\item[$(3)$] If $u\in J(I)$ and $p\in I$ with $up=u$, then $p\in J(I)$.
\item[$(4)$] If $u\in J(I)$ then $uI$ is a group with the neutral element $u$.
\item[$(5)$] If $p\in I$ then there is a unique $u\in J(I)$ with $up=p$.
\item[$(6)$] Let $u,v\in J(I)$ and let $p\in uI$. Then there is an $r\in I$ with $rp=v$ and $pr=u$.
\item[$(7)$] $I=\bigcup_{u\in J(I)}uI$.
\item[$(8)$] If $u,v\in J(I)$ with $u\not=v$, then $uI\cap vI=\emptyset$.
\item[$(9)$] Suppose $p\in E$ and $q,r\in I$ satisfy $qp=rp$. Then $q=r$.
\item[$(10)$] If $u\in J(I)$, then there is a unique $v\in J(I^\prime)$ such that $uv=v$ and $vu=u$.
\end{enumerate}
\end{lem}

\subsection{Generalizations of Auslander's theorem: abelian semigroup}
Using an approach different with Auslander's, the following is our first main result of this paper, in which the condition ``with a dense set of almost periodic points'' is essential in view of \cite[Example~3.25]{AD}.

\begin{thm}\label{thm2.15A}
Let $(T,X)$ be a semiflow with $T$ an abelian semigroup. If $(T,X)$ is weakly almost periodic with a dense set of almost periodic points, then it is uniformly almost periodic.
\end{thm}

\begin{proof}
Let $I$ be a minimal left ideal in the Ellis semigroup $E(X)$ of $(T,X)$. Since $(T,X)$ is \textit{w.a.p.}, by Lemma~\ref{lem2.6} we see $E(X)\subset C(X,X)$. First, we claim the following:
\begin{itemize}
\item[(i)] $I$ contains a unique idempotent, say $u$.

Indeed, if $u,v$ are idempotents in $I$, then by Lemma~\ref{lem2.14A}-(2) it follows that $uv=u$ so $uv=uu$. Let $\{t_n\}$ be a net in $T$ with $t_n\to u$ in $E(X)$; then
$$
uv=\lim t_nv=\lim t_nu=uu.
$$
But $t_nv=vt_n$ and $t_nu=ut_n$ for $T$ is abelian, thus
$$
uv=\lim vt_n=\lim ut_n=uu.
$$
Because $u,v\in C(X,X)$, $vu=\lim vt_n$ and $uu=\lim ut_n$. Therefore, $vu=uu$. Whence Lemma~\ref{lem2.14A}-(9) follows $v=u$. This proves our claim (i).
\end{itemize}

Next, for any almost periodic point $x$, $ux=x$ so that $ux=x$ for each $x\in X$, since $u\in C(X,X)$ and the almost periodic points are dense in $X$. Then $I=E(X)$ is a group of continuous maps of $X$ with the neutral element $\textit{id}_X$ by Lemma~\ref{lem2.14A}-(4).

Finally, by Ellis' theorem (i.e. Theorem~\ref{thm2.4} above), $E(X)$ and so $T$ both acts equicontinuously on $X$, and each $t\in T$ is a self-homeomorphism of $X$. Thus $(T,X)$ is an \textit{u.a.p.} semiflow by Theorem~\ref{thm1.6}.
\end{proof}

It should be noticed that \cite[Example~2.2]{AD} is a minimal, equicontinuous (and so \textit{w.a.p.}), but not uniformly almost periodic, semiflow on a discrete phase space with a non-abelian phase semigroup. This shows that the abelian condition in Theorem~\ref{thm2.15A} is also important.

\begin{cor}
Let $(T,X)$ be a minimal semiflow with $T$ an abelian semigroup. Then $(T,X)$ is weakly almost periodic if and only if it is uniformly almost periodic.
\end{cor}

\subsection{Generalizations of Auslander's theorem: group actions}
Our second main result of this paper is the following generalization of Auslander's Theorem~\ref{thm2.10A} in the non-minimal group-action case.

\begin{thm}\label{thm2.16A}
Let $(T,X)$ be a flow with phase group $T$ and with a dense set of almost periodic points. If $(T,X)$ is weakly almost periodic, then it is equicontinuous and so uniformly almost periodic.
\end{thm}

\begin{proof}
Let $I$ be a minimal left ideal in Ellis' semigroup $E(X)$ of $(T,X)$, where $E(X)\subset C(X,X)$ by Lemma~\ref{lem2.6}. We will first show the following claim.
\begin{itemize}
\item[(ii)] Let $(T,Y)$ be a weakly almost periodic semiflow with $\mathbb{I}$ a minimal left ideal in $E(Y)$. Then
\begin{gather*}
\Pi\colon T\times \mathbb{I}\rightarrow \mathbb{I},\quad (t,p)\mapsto tp=t\circ p,
\end{gather*}
is also a weakly almost periodic semiflow.

\begin{proof}
In fact, this is essentially Ellis-Nerurkar~\cite[Proposition~II.4]{EN}. We will show that $E(T,\mathbb{I})\subset C(\mathbb{I},\mathbb{I})$. For this, let $\phi\in E(T,\mathbb{I})$ with $t_n\to\phi$ where $\{t_n\}$ is a net in $T$. This means that $t_np\to\phi p$ for all $p\in \mathbb{I}$. Without loss of generality, let $t_n\to q\in E(T,Y)$. Then by $p\in C(Y,Y)$, we have $qp=\phi p$. This shows that $\phi\colon \mathbb{I}\rightarrow \mathbb{I}$ is the left-translation function $L_q\colon p\mapsto qp$ of $\mathbb{I}$ to itself. Since $q\in C(Y,Y)$, then $p_n\to p$ in $\mathbb{I}$ implies that $qp_n\to qp$ in $\mathbb{I}$ under the pointwise topology. Thus $\phi\in C(\mathbb{I},\mathbb{I})$.
\end{proof}
\end{itemize}
Now we apply (ii) to the case that $(T,Y)=(T,X)$ and $\mathbb{I}=I$ and then we can obtain the following claim.
\begin{itemize}
\item[(iii)] $\Pi\colon T\times I\rightarrow I$ is an equicontinuous flow. \quad (It does not need the density condition here.)

In fact, this follows at once from (ii) and Theorem~\ref{thm2.10A}.
\end{itemize}
Next, we can claim that
\begin{itemize}
\item[(iv)] There exists a unique idempotent $u$ in $I$.

If $u,v$ are idempotents in $I$, then by Lemma~\ref{lem2.14A}-(2) it follows that $u$ is proximal to $v$ for $(T,I,\Pi)$. But by (iii), $(u,v)$ is an almost periodic point for $(T,I\times I,\Pi)$. Thus $u=v$.
\end{itemize}
Therefore, by $u\in C(X,X)$, we can conclude that $ux=x$ for each $x\in X$. This shows that $I=E(X)$ is a group of continuous maps of $X$ into $X$ containing $\textit{id}_X$. Therefore, $(T,X)$ is an equicontinuous and so uniformly almost periodic flow by Theorem~\ref{thm2.4}.
\end{proof}

Notice that in Theorem~\ref{thm2.16A}, if assume additionally $(T,X)$ is \textit{topologically transitive} (i.e. for any non-empty open sets $U,V$ in $X$, $U\cap t^{-1}V\not=\emptyset$ for some $t\in T$; cf.~Definition~\ref{def2.22A}), then we will show that $(T,X)$ is minimal (see Corollary~\ref{c1} to Lemma~\ref{lem2.25}).

\subsection{Two genericity theorems}
Next we will consider a semiflow $(T,X)$ with each $t\in T$ a bijection of $X$. For this end we need some preliminaries.

Let $Y$ and $Z$ be topological spaces; then $C_{\mathfrak{p}}(Y,Z)$ will denote the set of all continuous maps from $Y$ to $Z$ provided with the topology of pointwise convergence; that is, a net $f_n\to g$ if and only if $f_n(y)\to g(y)$ for all $y\in Y$. If $(Z,\mathscr{U}_Z)$ is a uniform space where $\mathscr{U}_Z$ is a symmetric uniform structure, then $\mathscr{U}_{(Y,Z)}$ denotes the uniformity on $C(Y,Z)$ defined by the base $\{\tilde{\varepsilon}\,|\,\varepsilon\in\mathscr{U}_Z\}$, where $(f,g)\in\tilde{\varepsilon}$ if and only if $(f(y),g(y))\in\varepsilon$ for all $y\in Y$. We write $C_{\mathfrak{u}}(Y,Z)$ for this uniform space. So if $E\subset C(Y,Z)$ is compact in $C_{\mathfrak{u}}(Y,Z)$, then it is also compact in $C_{\mathfrak{p}}(Y,Z)$.

The following classical result, due to J.P.~Troallic, is our important tool.

\begin{lem}[{cf.~\cite[Lemma~4.2]{Aus}}]\label{lem2.17A}
Let $W$ be a locally compact T$_2$-space, $Y$ a compact T$_2$-space, and $Z$ a uniform space. Let $\varphi\colon W\rightarrow C_{\mathfrak{p}}(Y,Z)$ be a continuous function, and let $V\in\mathscr{U}_{(Y,Z)}$. Then there is a dense open subset $U$ of $W$ such that if $w\in U$, then $\varphi^{-1}[V[\varphi(w)]]$ is a neighborhood of $w$ in $W$.
\end{lem}

In Lemma~\ref{lem2.17A} the condition that $\varphi\colon W\rightarrow C_{\mathfrak{p}}(Y,Z)$ be continuous is equivalent to say that $\Phi\colon W\times Y\rightarrow Z,\ (w,y)\mapsto\varphi(w)(y)$ is separately continuous.

\begin{defn}[\cite{GH,E69,Aus,AD}]
Let $(T,X)$ be a semiflow.
\begin{enumerate}
\item We say $(T,X)$ is \textit{equicontinuous at a point} $x\in X$, write $x\in \textsl{Equi}\,(T,X)$, provided that for any $\varepsilon\in\mathscr{U}_X$, there is some $\delta\in\mathscr{U}_X$ such that $t(\delta[x])\subseteq\varepsilon[tx]$ for each $t\in T$.

\item[{}] Then $(T,X)$ is equicontinuous (cf.~Definition~\ref{def1.4}) if and only if $\textsl{Equi}\,(T,X)=X$.

\item Given any $\varepsilon\in\mathscr{U}_X$, we say $x\in\textsl{Equi}_\varepsilon(T,X)$ if and only if there is some $\delta\in\mathscr{U}_X$ such that $t(\delta[x])\subseteq\varepsilon[tx]$ for each $t\in T$.
\end{enumerate}

Clearly $\textsl{Equi}\,(T,X)=\bigcap_{\varepsilon\in\mathscr{U}_X}\textsl{Equi}_\varepsilon(T,X)$ and $\textsl{Equi}_\varepsilon(T,X)$ is an open subset of $X$. $\textsl{Equi}_\varepsilon(T,X)$ and $\textsl{Equi}\,(T,X)$ are invariant if $(T,X)$ is a flow. However, they are in general not invariant if $T$ is only a semigroup.
\end{defn}

Motivated by Auslander's proof of Theorem~\ref{thm2.10A} presented in \cite[p.~61]{Aus}, using Lemma~\ref{lem2.17A} we can obtain our first generic result which will be the starting point for some of our later arguments.

\begin{thm}\label{thm2.19A}
Let $(T,X)$ be a weakly almost periodic semiflow with phase semigroup $T$. Then for any $\varepsilon\in\mathscr{U}_X$, $\textsl{Equi}_\varepsilon(T,X)$ is a dense open set in $X$. In particular, if $X$ is a compact metric space, then $\textsl{Equi}\,(T,X)$ is a residual set in $X$.
\end{thm}

\begin{proof}
Let $\varphi\colon X\rightarrow C_{\mathfrak{p}}(E(X),X)$ be defined by $\varphi(x)(p)=p(x)$ for $x\in X$ and $p\in E(X)$, where $E(X)=E(T,X)$. Since $E(X)\subset C(X,X)$ by the \textit{w.a.p.} hypothesis, $\varphi$ is continuous.
Now, let $\varepsilon\in\mathscr{U}_X$ and let $V=\tilde{\varepsilon}\in\mathscr{U}_{(E(X),X)}$ so
\begin{gather*}
V=\big{\{}(f,g)\in C(E(X),X)\times C(E(X),X)\,|\,(f(p),g(p))\in\varepsilon\ \forall p\in E(X)\big{\}}.
\end{gather*}
By Lemma~\ref{lem2.17A}, there is a dense open set $U$ in $X$ such that if $x\in U$, then $N_x=\varphi^{-1}[V[\varphi(x)]]$ is a neighborhood of $x$. So, if $x^\prime\in N_x$, $(\varphi(x),\varphi(x^\prime))\in V$, or that is the same thing, $(p(x),p(x^\prime)\in\varepsilon$ for all $p\in E(X)$. Next, we take some $\delta\in\mathscr{U}_X$ with $\delta[x]\subseteq N_x$. Then for any $t\in T$, $t(\delta[x])\subseteq\varepsilon[tx]$. Thus $U\subseteq\textsl{Equi}_\varepsilon(T,X)$. This proves Theorem~\ref{thm2.19A}.
\end{proof}

Our another generic result is the following, which will be needed for our later Theorem~\ref{thm2.34A}.

\begin{thm}\label{thm2.20A}
Let $\varphi,\varphi_n, n=1,2,\dotsc$ be continuous self-maps of $X$ with $\varphi_n\to\varphi$ pointwise on $X$. Then for any $\varepsilon\in\mathscr{U}_X$, there exists a dense open $E_\varepsilon$ of $X$ such that for any $x_0\in E_\varepsilon$, one can find a neighborhood $V$ of $x_0$ and $n_0>0$ so that $(\varphi_n(x),\varphi_n(y))\in\varepsilon$ for all $x,y\in V$ and $n\ge n_0$.
\end{thm}

\begin{proof}
Let $\varepsilon,\beta,\delta\in\mathscr{U}_X$ be any given with $\delta^3\subset\beta$ and $\beta^3\subset\varepsilon$. For any integer $n\ge1$, set
\begin{gather*}
K_n=\left\{z\in X\,|\,(\varphi(z),\varphi_m(z))\in\delta\ \forall m\ge n\right\}.
\end{gather*}
Since $\varphi_n\to\varphi$ pointwise on $X$, $X=\bigcup_{n\ge1}K_n$ and so $X=\bigcup_{n\ge1}\overline{K}_n$. Thus for some $n_0\ge1$, $\overline{K}_{n_0}$ has nonempty interior $U$. Now for any $x\in U$ and any $m\ge n_0$, we have $(\varphi(x),\varphi_m(x))\in\beta$. Indeed, since $\varphi$ and $\varphi_m$ both are continuous functions on $X$, as $K_{n_0}\ni z\to x$, we have $(\varphi(x),\varphi(z))\in\delta$, $(\varphi_m(x),\varphi_m(z))\in\delta$ and $(\varphi(z),\varphi_m(z))\in\delta$ whence $(\varphi(x),\varphi_m(x))\in\beta$. By the continuity of $\varphi$ again, it follows that for any $x_0\in U$, one can find a neighborhood $V$ of $x_0$ such that
\begin{gather*}
(\varphi_m(x),\varphi_m(y))\in\varepsilon\quad \forall x,y\in V\textrm{ and }m\ge n_0.
\end{gather*}
By considering $\overline{W}$ of any nonempty open subset $W$ of $X$ instead of $X$ in the above argument, we can conclude Theorem~\ref{thm2.20A}.
\end{proof}

\subsection{Generalizations of Auslander's theorem: \textit{C}-semigroups}
As the first application of Theorem~\ref{thm2.19A}, we will consider a semiflow with a \textit{C}-semigroup as our phase semigroup.

\begin{defn}[\cite{KM}]
A topological semigroup $T$ is called a \textit{right C-semigroup} if for any $s\in T$, $T\setminus{Ts}$ is relatively compact in $T$.

Each topological group is of course a right \textit{C}-semigroup. Under the usual topology, $(\mathbb{Z}_+,+)$ and $(\mathbb{R}_+,+)$ both are right \textit{C}-semigroups.
\end{defn}

\begin{defn}\label{def2.22A}
Let $(T,X)$ be a semiflow on any topological space $X$.
\begin{enumerate}
\item $(T,X)$ is called \textit{topologically transitive} if for all non-empty open sets $U,V$ in $X$, $U\cap t^{-1}V\not=\emptyset$ for some $t\in T$.

\item A point $x$ of $X$ is called a \textit{transitive point}, denoted by $x\in\textsl{Trans}\,(T,X)$, if $\textrm{cls}_XTx=X$. If $\textsl{Trans}\,(T,X)$ is non-empty, then $(T,X)$ is called \textit{point transitive}.
\end{enumerate}

It should be noted that in our situation, the topologically transitive $\not\Leftrightarrow$ the point transitive, for a semiflow $(T,X)$ with $T$ not a group. See \cite{AD}.
\end{defn}

\begin{lem}[{cf.~\cite[Lemma~4.11]{AD}}]\label{lem2.23A}
Let $(T,X)$ be a semiflow with $T$ a right \textit{C}-semigroup. If $\textsl{Equi}_\varepsilon(T,X)\not=\emptyset$ for any $\varepsilon\in\mathscr{U}_X$, then $\textsl{Trans}\,(T,X)\subseteq\textsl{Equi}\,(T,X)$.
\end{lem}

The following is another main result of this paper, which generalizes Auslander's Theorem~\ref{thm2.10A} for any topological group is a right \textit{C}-semigroup.

\begin{thm}\label{thm2.24A}
Let $(T,X)$ be a minimal weakly almost periodic semiflow with $T$ a right \textit{C}-semigroup. Then $(T,X)$ is equicontinuous.
\end{thm}

\begin{proof}
First, $\textsl{Trans}\,(T,X)=X$. By Theorem~\ref{thm2.19A}, $\textsl{Equi}_\varepsilon(T,X)\not=\emptyset$ for $\varepsilon\in\mathscr{U}_X$. Thus by Lemma~\ref{lem2.23A}, $\textsl{Equi}\,(T,X)=X$ and so $(T,X)$ is equicontinuous.
\end{proof}

\begin{corr}\label{cor1}
Let $(T,X)$ be a pointwise almost periodic and weakly almost periodic semiflow with $T$ an amenable right \textit{C}-semigroup. Then $(T,X)$ is a uniformly almost periodic semiflow.
\end{corr}

\begin{proof}
First of all, if $(T,X)$ is a minimal, then $(T,X)$ is a uniformly almost periodic semiflow.
Indeed, by Theorem~\ref{thm2.24A}, $(T,X)$ is minimal equicontinuous. Moreover, by Lemma~\ref{lem2.11A}, each $t\in T$ is a surjection of $X$ so $(T,X)$ is uniformly almost periodic from Theorem~\ref{thm1.6}.

Now assume that $(T,X)$ is pointwise almost periodic.
Let $I$ be a minimal left ideal in $E(T,X)$; then by (ii) in the proof of Theorem~\ref{thm2.16A}, $(T,I,\Pi)$ is a minimal \textit{w.a.p.} semiflow. Then $(T,I,\Pi)$ is \textit{u.a.p.} and so it is a distal semiflow by Theorem~\ref{thm2.5}.

Let $x_0\in X$ be any given and define a continuous map $\phi\colon I\rightarrow\textrm{cls}_XTx_0$ by $p\mapsto p(x_0)$. Clearly, $(T,I)\xrightarrow{\phi}(T,\textrm{cls}_XTx_0)$ is an epimorphism so that each point of $\textrm{cls}_XTx_0$ is a distal point for $(T,X)$. Thus $(T,X)$ is a distal \textit{w.a.p.} semiflow. Then by Lemma~\ref{lem2.3}, $E(X)$ is a group of continuous self-maps of $X$. Further by Theorem~\ref{thm2.4}, $(T,X)$ is equicontinuous with each $t\in X$ a self-homeomorphism of $X$. Thus $(T,X)$ is \textit{u.a.p.} by Theorem~\ref{thm1.6}.
\end{proof}

Similarly we can obtain the following result.

\begin{corr}
Let $(T,X)$ be a pointwise almost periodic and weakly almost periodic semiflow with $T$ a right \textit{C}-semigroup such that each $t\in T$ is a self-surjection of $X$. Then $(T,X)$ is a uniformly almost periodic semiflow.
\end{corr}

It should be noticed that under the context of Theorem~\ref{thm2.24A}, we do not know if $(T,X)$ is uniformly almost periodic if without the amenability. In addition, Theorem~\ref{thm1.6} plays a role in the proof of above Corollary~\ref{cor1} for $T$ need not be a normal subgroup of $E(X)$.

\begin{defn}[\cite{DT,AD}]
A semiflow $(T,X)$ is referred to as a \textit{syndetically transitive semiflow} in case
$$N_T(U,V)=\{t\in T\,|\,U\cap t^{-1}V\not=\emptyset\}$$
is syndetic in $T$ for all non-empty open subsets $U,V$ of $X$.
\end{defn}

If $(T,X)$ is topologically transitive with a dense set of almost periodic points, then it is syndetically transitive (cf.~\cite[Lemma~4.2]{AD} and \cite{DT}).

\begin{defn}[\cite{KM,AD}]
A semiflow $(T,X)$ is said to be \textit{sensitive} if there exists an $\varepsilon\in\mathscr{U}_X$ such that for every $x\in X$ and any $\delta\in\mathscr{U}_X$, one can find some $y\in\delta[x]$ and $t\in T$ with $ty\not\in\varepsilon[tx]$.
\end{defn}

\begin{lem}[{cf.~\cite[Lemma~1.43]{AD}}]\label{lem2.24}
A semiflow $(T,X)$ is not sensitive iff $\textsl{Equi}_\varepsilon(T,X)\not=\emptyset$ for all $\varepsilon\in\mathscr{U}_X$.
\end{lem}

\begin{lem}[{\cite[Lemma~4.5]{AD}}]\label{lem2.25}
Let $(T,X)$ be a non-minimal syndetically transitive semiflow. Then $(T,X)$ is sensitive.
\end{lem}

\begin{corollary}\label{c1}
Every weakly almost periodic syndetically transitive semiflow is minimal. Hence if $(T,X)$ is a weakly almost periodic syndetically transitive flow, then it is minimal uniformly almost periodic.
\end{corollary}

\begin{proof}
Let $(T,X)$ be a \textit{w.a.p.} syndetically transitive semiflow. By Theorem~\ref{thm2.19A} and Lemma~\ref{lem2.24}, $(T,X)$ is not sensitive. Thus $(T,X)$ is minimal by Lemma~\ref{lem2.25}. The second part follows at once from the first part and Theorem~\ref{thm2.16A}.
\end{proof}

\begin{corollary}\label{c2}
If $(T,X)$ is a weakly almost periodic semiflow, then it is not sensitive.
\end{corollary}

\begin{proof}
This follows easily from Theorem~\ref{thm2.19A} and Lemma~\ref{lem2.24}.
\end{proof}

We will need the following dichotomy theorem of syndetically transitive semiflows.

\begin{lem}[{cf.~\cite[Theorem~4.12]{AD}}]\label{lem2.26}
Let $(T,X)$ be a syndetically transitive semiflow with $T$ a right \textit{C}-semigroup. Then $(T,X)$ is either equicontinuous or sensitive.
\end{lem}

Then we can obtain another main result of this paper.

\begin{thm}\label{thm2.27}
Let $(T,X)$ be a syndetically transitive semiflow with $T$ a right \textit{C}-semigroup. If $(T,X)$ is weakly almost periodic, then it is minimal equicontinuous.
\end{thm}

\begin{proof}
By Corollary~\ref{c2} to Lemma~\ref{lem2.25}, $(T,X)$ is not sensitive. So $(T,X)$ is equicontinuous by Lemma~\ref{lem2.26}.
Further by Corollary~\ref{c1} to Lemma~\ref{lem2.25}, $(T,X)$ is minimal equicontinuous.
\end{proof}
\subsection{Generalizations of Auslander's theorem: amenable semigroups}
We note that $(\mathbb{R}_+^n,+), n\ge2$ is an amenable semigroup but not a right \textit{C}-semigroup. Thus Theorem~\ref{thm2.15A} and the following result are beyond Theorems~\ref{thm2.24A} and \ref{thm2.27} above.

\begin{thm}\label{thm2.31A}
Let $(T,X)$ be a minimal weakly almost periodic semiflow with $T$ an amenable semigroup such that each $t\in T$ is self-bijection of $X$. Then $(T,X)$ is uniformly almost periodic.
\end{thm}

\begin{proof}
In the context of this theorem, it is known that
\begin{itemize}
\item If $\textsl{Equi}_\varepsilon(T,X)\not=\emptyset$ for any $\varepsilon\in\mathscr{U}_X$, then $(T,X)$ is equicontinuous (cf.~\cite[Lemma~1.29]{AD}).
\end{itemize}
By Theorem~\ref{thm2.19A}, $\textsl{Equi}_\varepsilon(T,X)\not=\emptyset$ for any $\varepsilon\in\mathscr{U}_X$ so $(T,X)$ is equicontinuous. Theorem~\ref{thm1.6} follows that $(T,X)$ is an \textit{u.a.p.} semiflow.
\end{proof}

Comparing with Theorem~\ref{thm2.15A}, here ``amenable'' is weaker than ``abelian'', but we need to require previously that each $t\in T$ is a self-homeomorphism of $X$ in Theorem~\ref{thm2.31A}. However, this bijection is just a conclusion of Theorem~\ref{thm2.15A}.

\subsection{Generalizations of Auslander's theorem: universally transitive semiflows}
Using the universally transitive in place of the amenability or \textit{C}-semigroup condition, we can obtain the following sufficient and necessary condition.

\begin{thm}\label{thm2.32A}
Let $T$ be a semigroup of continuous self-maps of $X$ such that $(T,X)$ is minimal and universally transitive. Then $(T,X)$ is weakly almost periodic if and only if $(T,X)$ is equicontinuous if and only if $(T,X)$ is uniformly almost periodic.
\end{thm}

\begin{proof}
Since $(T,X)$ is a universally transitive semiflow, then it is homogeneous and so each $t\in T$ is a surjection of $X$ (cf.~\cite[Proposition~2.6]{AD}). Thus $(T,X)$ is equicontinuous if and only if $(T,X)$ is \textit{u.a.p.} by Theorem~\ref{thm1.6}.
We only need prove the necessity. For this, assume $(T,X)$ is \textit{w.a.p.} and we will show that $(T,X)$ is distal.

Indeed, we first note that $\mathrm{Aut}\,(T,X)$ acts freely on $X$ (i.e. if $a\in\mathrm{Aut}\,(T,X)$ and $a\not=\textit{id}_X$, then $a(x)\not=x\ \forall x\in X$. In fact, let $ax_0=x_0$ for some $x_0\in X$; then $atx_0=tx_0$ for all $t\in T$ so $ax=x$ for all $x\in X$ for $\textrm{cls}_XTx_0=X$). Now for any $x,y\in X, x\not=y$, one can find some $a\in\mathrm{Aut}\,(T,X)$ with $y=ax$. If a net $\{t_n\}$ in $T$ is such that $\lim_nt_nx=z=\lim_nt_ny$, then $az=z$ whence $a=\textit{id}_X$. This contradicts $x\not=y$. Thus $T$ acts distally on $X$. Consequently, $E(T,X)$ is a group of continuous self-maps of $X$ by Lemma~\ref{lem2.3} and so $(T,X)$ is equicontinuous by Theorem~\ref{thm2.4}.
\end{proof}

We notice here that Theorem~\ref{thm2.32A} is comparable with the following known result for which the phase space is metrizable.

\begin{lem}[{\cite[Theorem~1.31]{AD}}]\label{lem2.33A}
If $(T,X)$ is a universally transitive semiflow on a compact metric space, then it is uniformly almost periodic.
\end{lem}

Of course, the consequence of Lemma~\ref{lem2.33A} is false when $X$ is non-metric in view of Ellis' two circle minimal system~\cite{E69}. In addition, we will consider the converse of this lemma in Theorem~\ref{thm2.36} in $\S\ref{sec2.12}$.

\begin{thm}\label{thm2.34A}
Let $(T,X)$ be a universally transitive semiflow such that $\textrm{Aut}\,(T,X)$ is equicontinuous, where $X$ is a compact, first countable, T$_2$-space. Then $(T,X)$ is weakly almost periodic if and only if it is uniformly almost periodic.
\end{thm}

\begin{proof}
By Theorem~\ref{thm1.6}, the sufficiency is obvious. Conversely, let $E(T,X)\subseteq C_\mathfrak{p}(X,X)$ by Lemma~\ref{lem2.6}; that is, $E(T,X)$ be a compact subspace of $C_\mathfrak{p}(X,X)$. Let $x_0\in X$ and then there is a countable neighborhood base at $x_0$.

Then given any net $\{t_n\}$ in $T$, we can find a subsequence $\{t_i\}_{i=1}^\infty$ from this net with $t_ix_0\to\varphi(x_0)$ for some $\varphi\in C(X,X)$. Further by $\mathrm{Aut}(T,X)x_0=X$, it follows that $t_i\to\varphi$ pointwise on $X$. Next, we shall proceed to prove that $t_i\to\varphi$ uniformly on $X$.

First let $\varepsilon\in\mathscr{U}_X$ be any given. Since $\mathrm{Aut}(T,X)$ is equicontinuous on $X$ by hypothesis, there exists $\delta\in\mathscr{U}_X$ such that $a[\delta]\subseteq\varepsilon$ for each $a\in\mathrm{Aut}(T,X)$. Moreover, by Theorem~\ref{thm2.20A}, there is a point $y\in X$ such that $t_i\to\varphi$ $\delta$-uniformly at the point $y$; that is, one can find a neighborhood $V_y$ of $y$ and an integer $i_0>0$ such that $(t_ix, t_iz)\in\delta$ for all $i\ge i_0$ and $x,z\in V_y$.

Let $x\in X$ be any given and let $a\in\mathrm{Aut}(T,X)$ with $ay=x$. Let $U=a[V_y]$ which is a neighborhood of $x$. Then, for any $z\in U$ and $i>i_0$,  $(t_ix,t_iz)\in\varepsilon$. Therefore, $t_i\to\varphi$ uniformly at every point $x\in X$ and then on $X$.
Thus $(T,X)$ is equicontinuous by the Ascoli-Arzel\`{a} theorem. Further by Theorem~\ref{thm1.6}, $(T,X)$ is uniformly almost periodic.
\end{proof}

\subsection{A note of a theorem of Gottschalk}\label{sec2.12}
We will now generalize the following theorem of Gottschalk.

\begin{thm}[{\cite[Theorem~7]{G}}]
Let $(G,X)$ be a point transitive flow on a compact metric space with $G$ an abelian group; then the following statements are pairwise equivalent:
\begin{enumerate}
\item[$(1)$] $(G,X)$ is universally transitive.
\item[$(2)$] $\mathrm{Aut}\,(G,X)$ is equicontinuous on $X$.
\item[$(3)$] $(G,X)$ is equicontinuous.
\end{enumerate}
\end{thm}

\begin{thm}\label{thm2.36}
Let $(T,X)$ be a point transitive semiflow on a compact metric space with $T$ an abelian semigroup; then the following statements are pairwise equivalent:
\begin{enumerate}
\item[$(1)$] $(T,X)$ is universally transitive.
\item[$(2)$] $\mathrm{Aut}\,(T,X)$ is equicontinuous on $X$ such that each $t\in T$ is a self-homeomorphism of $X$.
\item[$(3)$] $(T,X)$ is equicontinuous and each $t\in T$ is a self-surjection of $X$.
\item[$(4)$] $(T,X)$ is uniformly almost periodic.
\end{enumerate}

\begin{note}
Condition $(1)\Rightarrow(4)\Rightarrow(3)\Rightarrow(2)$ in our proof does not need the hypothesis that $T$ is abelian, but $(2)\Rightarrow(1)$ needs this.
\end{note}
\end{thm}

\begin{proof}
Condition $(1)\Rightarrow(4)$: This follows immediately from Lemma~\ref{lem2.33A}.

Condition $(4)\Rightarrow(3)$: This is exactly Theorem~\ref{thm1.6}.

Condition $(3)\Rightarrow(2)$: By Theorem~\ref{thm2.5}, each $t\in T$ is a self-homeomorphism of $X$. Then $E(T,X)$ is a compact topological group of homeomorphisms of $X$ by Lemma~\ref{lem2.3} and condition (3). It is easy to see that $ap=pa$ for all $a\in\mathrm{Aut}\,(T,X)$ and $p\in E(T,X)$. Thus $E(T,X)\subseteq\mathrm{Aut}\,(G,X)$, where $G=\mathrm{Aut}\,(T,X)$ and we regard $(G,X)$ as a flow on $X$ with phase group $G$. Since $(T,X)$ is transitive, thus $E(T,X)x=\textrm{cls}_XTx=X$ for some and then all $x\in X$. Then $(G,X)$ is a universally transitive flow by Definition~\ref{def2.7}.2. So by Lemma~\ref{lem2.33A}, it follows that $\mathrm{Aut}\,(T,X)$ is equicontinuous acting on $X$. Hence (2) holds.

Condition $(2)\Rightarrow(1)$: Since $T$ is abelian, then $E(T,X)\subset\mathrm{Aut}\,(T,X)$ so $(T,X)$ is a universally transitive semiflow.

The proof of Theorem~\ref{thm2.36} is therefore completed.
\end{proof}
\section{Unique ergodicity of some semiflows}
\begin{defn}
Let $(T,X)$ be a semiflow on $X$; then a Borel probability measure $\mu$ on $X$ is called \textit{invariant} if $\mu(B)=\mu(t^{-1}B)$ for all Borel subsets $B$ of $X$. If there is a unique such measure, then we say $(T,X)$ is \textit{uniquely ergodic}.

It is a well-known result that any semiflow with amenable phase semigroup admits invariant Borel probability measures. The proof of the above Theorem~\ref{thm2.31A} needs essentially this important property of amenable semigroups.
\end{defn}

In \cite{EN}, by using the closed invariant equivalent proximal relation
\begin{equation*}
P(T,X)=\{(x,y)\in X\times X\,|\,\exists p\in E(X)\textrm{ s.t. }p(x)=p(y)\},
\end{equation*}
Ellis and Nerurkar proved the following unique ergodicity, which will be generalized in a little later.

\begin{thm}[{cf.~\cite[Proposition~II.10]{EN}}]\label{thm3.2}
If $(T,X)$ is a weakly almost periodic flow with a unique minimal set, then it is uniquely ergodic.
\end{thm}

We can first generalize this theorem as follows:

\begin{cor}\label{thm3.3}
Let $(T,X)$ be a weakly almost periodic flow. Then, restricted to any minimal subset $\Theta$ of $X$, $(T,\Theta)$ is minimal equicontinuous and hence it is uniquely ergodic.
\end{cor}

\begin{proof}
Since $(T,X)$ is \textit{w.a.p.}, then $(T,\Theta)$ is \textit{w.a.p.} so it is minimal equicontinuous. Thus $(T,\Theta)$ is uniquely ergodic by Theorem~\ref{thm3.2}.
\end{proof}

\begin{proof}
Let $I$ be a minimal left ideal in $E(X)$; then $(T,I,\Pi)$ is minimal equicontinuous by (iii) in the proof of Theorem~\ref{thm2.16A}. Now given any $x_0\in\Theta$, let $\varphi\colon I\rightarrow \Theta$ be given by $p\mapsto p(x_0)$. Clearly, $(T,I)\xrightarrow{\varphi}(T,\Theta)$ be an epimorphism so $(T,\Theta)$ is equicontinuous. Then by Theorem~\ref{thm3.2}, $(T,\Theta)$ is uniquely ergodic.
The proof of Theorem~\ref{thm3.3} is thus completed.
\end{proof}

Given any semigroup $T$ of self-homeomorphisms of $X$, by $\langle T\rangle$ we denote the group generated by $T$. Then:

\begin{lem}[{\cite[Theorem~3.12]{AD}}]\label{lem3.4}
Let $(T,X)$ be a minimal equicontinuous semiflow with each $t\in T$ a self-homeomorphism of $X$. Then $(\langle T\rangle,X)$ is a minimal equicontinuous flow.
\end{lem}

\begin{thm}\label{thm3.5}
Let $(T,X)$ be a minimal weakly almost periodic semiflow with $T$ an amenable semigroup such that each $t\in T$ is bijection of $X$. Then $(T,X)$ is uniquely ergodic.
\end{thm}

\begin{proof}
Since $T$ is amenable, there exists at least one invariant Borel probability measure for $(T,X)$. Let $\mu$ and $\nu$ be two invariant Borel probability measures for $(T,X)$.

Next by Theorem~\ref{thm2.31A}, $(T,X)$ is minimal equicontinuous. Then by Lemma~\ref{lem3.4}, $(\langle T\rangle,X)$ is a minimal equicontinuous flow. Clearly, $\mu, \nu$ are invariant for $(\langle T\rangle,X)$. Thus $\mu=\nu$ by Theorem~\ref{thm3.2}. Therefore $(T,X)$ is uniquely ergodic.
\end{proof}

Similarly we can obtain the following. We omit its proof here.

\begin{thm}\label{thm3.6}
Let $(T,X)$ be a minimal weakly almost periodic semiflow with $T$ a right \textit{C}-semigroup such that each $t\in T$ is self-surjection of $X$. Then $(T,X)$ is uniquely ergodic.
\end{thm}
\section*{\textbf{Acknowledgments}}%
The author cordially thanks Professor Joe~Auslander for his much helpful comments on this note.

This work was partly supported by National Natural Science Foundation of China (Grant Nos. 11431012, 11271183) and PAPD of Jiangsu Higher Education Institutions.

\end{document}